\documentclass[11pt]{article}
\usepackage{amsmath, amsfonts, amssymb,  mathrsfs,  array, stmaryrd,  indentfirst, amsthm,  hyperref, comment}
\usepackage{graphicx, enumitem, tabularx, color, enumitem}
\usepackage{tcolorbox}
\usepackage{float} 

\usepackage[margin=1.0in]{geometry}
\usepackage{setspace}

\numberwithin{equation}{section}
\theoremstyle{plain}

\newtheorem{thm}{Theorem}[section]
\newtheorem{defn}{Definition}[section]
\newtheorem{prop}{Proposition}[section]

\newtheorem{remark}{Remark}[section]
\newtheorem{example}{Example}[section]

\newcommand{\bbT}{\mathbb{T}}

\newcommand{\bbW}{\mathbb{W}}
\newcommand{\bbR}{\mathbb{R}}

\makeatletter
\def\@setcopyright{}
\def\serieslogo@{}
\makeatother

\begin{document}
\title{\textbf{Terminal Ranking Games}\thanks{
Erhan Bayraktar is supported in part by the NSF under grant DMS-1613170 and by the Susan M. Smith Professorship. We are grateful to Jak\v sa Cvitani\'c for stimulating discussions.}}

 \author{Erhan Bayraktar\thanks{Department of Mathematics, University of Michigan, 530 Church Street, Ann Arbor, MI 48104, USA, erhan@umich.edu} \and  Yuchong Zhang\thanks{Department of Statistical Sciences, University of Toronto, 100 St.\ George Street, Toronto, Ontario M5S 3G3, Canada, yuchong.zhang@utoronto.ca}
 }
 \date{\today}
 
\maketitle

\begin{abstract}
We analyze a mean field tournament: a mean field game in which the agents receive rewards according to the ranking of the terminal value of their projects and are subject to cost of effort. Using Schr\"{o}dinger bridges we are able to explicitly calculate the equilibrium. This allows us to identify the reward functions which would yield a desired equilibrium and solve several related mechanism design problems. We are also able to identify the effect of reward inequality on the players' welfare as well as calculate the price of anarchy. 



\end{abstract}

\section{Introduction}

Consider the following tournament: each player (indexed by $i\in\{1,\ldots, N\}$) exerts an effort, which we denote by $a_i$, to move the value of her project/state, which is modeled as a drifted Brownian motion:
\[X_{i,t}=x_{0}+\int_0^t a_{i,s}ds+\sigma B_{i,t}.\]
We assume $B_1, \ldots, B_N$ are independent. 
The cost of effort per unit time is assumed to be quadratic in $a_{i}$ with coefficient $c$. The game ends at time $T>0$, when each player receives a reward that is a deterministic function of three components:
\begin{itemize}[noitemsep]
\item Her terminal project value $X_{i,T}$; 
\item The ranking of $X_{i,T}$ relative to other players, measured by the fraction $\frac{1}{N}\sum_{j=1}^N 1_{\{X_{j,T}\le X_{i,T}\}}$ of players having equal or worse performance (so that the top performer has rank one and the bottom performer has rank $1/N$); 
\item Statistics of the population performance, such as population mean $\frac{1}{N}\sum_{j=1}^N X_{j,T}$ or the $k$-th order statistic of $X_{1,T},\ldots, X_{N,T}$ or both. This allows us to cover the case when the ``reward pie" is not fixed, but grows with the total production or the $k$-th best performance. 
For simplicity of the presentation, we only consider dependence via the population mean.
\end{itemize}

In this paper, we will analyze the \emph{mean field game} associated with the above $N$-player game, and explicitly characterize the equilibrium (see Section~\ref{sec:NE}), improving on the results of Bayraktar and Zhang \cite{BZ16} which dealt only with the abstract existence and uniqueness of the mean-field equilibrium. Analysis of mean-field games is useful in solving $N$-player games when $N$ is large, since it has been shown in Bayraktar and Zhang \cite{BZ16} that the mean-field equilibrium can be used to construct an approximate Nash equilibria for the finite player games. 

Our explicit characterization, which is rare in mean field games, allows us to solve \emph{tournament design problems}. Specifically, we determine in Section~\ref{sec:tourn} the reward functions that maximize the rank-$\alpha$ performance, the net profit (for the tournament planner), and the total effort, respectively. Moreover, in Section~\ref{sec:PoA}, we also compute the so-called \emph{price of anarchy} which measures the efficiency loss due to decentralization; see e.g. Lacker and Ramanan \cite{doi:10.1287/moor.2018.0929}, Carmona et al.\ \cite{refId0}, and Cardaliaguet and Rainer \cite{CRSICON}.

Mean field games, introduced simultaneously by  Lasry and Lions \cite{LasryLions.06a,LasryLions.06b,LasryLions.07}, and Huang, Caines and Malham\'{e} \cite{HuangMalhameCaines.06, HuangMalhameCaines.07} (see also the two-volume book of Carmona and Delarue \cite{CarmonaDelaRue.17ab} for an extensive overview), analyze games with a large number of players which are weakly interacting through their empirical distribution. The main appeal of the mean field games is the decentralized structure of their equilibria: agents compute their best response to a given population distribution, which is then determined by a fixed point problem. The best response calculation is a pure stochastic control problem. Instead of working with the Hamilton-Jacobi-Bellman equation, we perform the calculation using Schr\"odinger bridges which can be seen as the stochastic analogue of quadratic optimal transport. (See L\'{e}onard \cite{Leonard14} and Chen et al.\ \cite{CGP16} for an overview of Schr\"odinger bridges and their connection to optimal transport.) We first introduce an auxiliary terminal distribution for the state (to convert the problem to an optimal transport problem), and then optimize over all such terminal distributions.
This approach allows us to reformulate the best response problem as a static calculus of variation problem,
which we then explicitly solve. This leads us to the next stage, the fixed point equation, whose solutions can be explicitly determined through its quantiles.

The distinguishing feature of our mean field game, i.e., tournament, is the rank-based feature of the reward. In particular, each player is rewarded according to the ranking of the terminal value of their project relative to the population, subject to cost of effort. This makes the analysis of the problem more difficult since the mean field interaction is non-local in the measure and the rank function is not regular. This problem was suggested by Gu\'{e}ant et al.\ \cite{GueantLasryLions.11} as a model in oil production, analyzed using abstract tools in the weak formulation by Carmona and Lacker \cite{LackerAAP} and in the strong formulation by Bayraktar and Zhang \cite{BZ16}. In these works continuity with respect to the rank was assumed. Related tournament games where the players are ranked according to their completion times have been considered by Bayraktar et al.\ \cite{BCZ18} for controlled Brownian motion dynamics and by Nutz and Zhang \cite{MZ17} for one-stage Poisson dynamics with controlled jump intensity. In the Appendix we are going to construct an extension of Schr\"{o}dinger bridges from space to time which can then be applied to construct the equilibrium in Bayraktar et al.\ \cite{BCZ18} as well. 

In economics there is a substantial literature on tournaments, going back to Lazear and Rosen \cite{LazearRosen81}; see Bayraktar et al.\ \cite{BCZ18} for a review. Most of these works focus on finitely many players or static models. Using such a one shot model, Fang et al.\ \cite{FNS18} analyze the discouraging effects of inequality. In our paper we observe a similar phenomenon, in that the more unequal the reward is (in the \emph{Lorenz order}, see e.g.\ Marshall et al.\ \cite{MajorizationTheory}) the smaller the game value for each player. However, unlike in the work of Fang et al.\ \cite{FNS18}, the same is not true for the effort in our set-up: the most fair distribution induces agents to put forth zero effort. Hence, one of the questions in the mechanism design section we investigate is what reward function maximizes the cumulative effort. We also analyze the case when agents have a social planner doing the optimization, which is used in computing the price of anarchy.  

The rest of the paper is organized as follows: In Section~\ref{sec:single} we consider the single player's problem and find her best response using Schr\"{o}dinger bridges. We then explicitly compute the mean field equilibrium in Section~\ref{sec:NE} and show the effect of reward inequality on the well-being of the players in Section~\ref{sec:Strack}. Section~\ref{sec:tourn} is where we investigate the tournament design problems with respect to several criteria. In Section~\ref{sec:PoA} we compute the price of anarchy. Finally, in Appendix~\ref{sec:Appen} we show how one can adapt the Schr\"odinger bridge approach to the completion time ranking game of Bayraktar et al.\ \cite{BCZ18}.

\section{A single player's problem}\label{sec:single}

Let us first describe the incentives of the player: We call $R(x,r,m):\bbR \times [0,1]\times \bbR \rightarrow \bbR\cup\{\pm\infty\}$ a
reward function if it is increasing in all of its arguments\footnote{Throughout the paper, increasing and decreasing are understood in the weak sense.}, $\bbR$-valued if $r\in(0,1)$, and satisfies $\int_0^1 R(x,r,m)dr<\infty$ for all $(x,m)\in\bbR^2$.
Denote the set of reward functions by $\mathcal{R}$ and the set of bounded reward functions by $\mathcal{R}_b$.

Given the distribution $\tilde \mu\in\mathcal{P}(\bbR)$ of the terminal project value of the population, we wish to find the best response to $\tilde \mu$ for a representative player. For any $\mu\in\mathcal{P}(\bbR)$, write $F_\mu$ for the cumulative distribution function (c.d.f.) of $\mu$ and $R_\mu(x)$ for $R(x, F_\mu(x), \int_{\bbR} y d\mu(y))$. A representative player with cost parameter $c$ solves the following stochastic control problem:
\begin{equation}\label{eq:original-control-prob}
\sup_{a} E\left[R_{\tilde \mu}(X_T)-\int_0^T c a_t^2 dt \right] \quad \text{where}\quad X_t= x_0+\int_0^t a_s ds+\sigma B_t.
\end{equation}
Here $a$ is admissible if it is progressively measurable and satisfies $E\int_0^T |a_s| ds<\infty$. 
%
Different from Bayraktar and Zhang \cite{BZ16}, let us consider the weak formulation of the above problem, which has some interesting connection with optimal transport.

Let $\Omega=C([0,T],\mathbb{R})$ be the Wiener space and $\mathbb{W}_x$ be the Wiener measure starting at $x$ at $t=0$. Under $\bbW_0$, the canonical process $\omega_t$ is a Brownian motion, and thus $X_{t}:=x_0+\sigma \omega_t$ represents the project value process under zero effort. Let $P$ be the law of $X$ under $\bbW_0$, and $(\mathcal{F}_t)_{t\in[0,T]}$ be the filtration generated by the process $\omega_t$. 

For any $Q\in \mathcal{P}(\Omega)$ such that $Q\sim P$, the Girsanov theorem implies that we can find an adapted process $a_t$ such that 
\[dX_t=a_t dt+\sigma dB^Q_t\] 
for some $Q$-Brownian motion $B^Q$. Conversely, given any sufficiently integrable adapted process $a_t$, we can define $Q\sim P$ such that the above equation holds. This means that if we restrict ourselves to sufficiently integrable effort process, we can identify $\mathcal{Q}:=\{Q\in\mathcal{P}(\Omega): Q\sim P\}$ with the set of laws of the controlled project value process $X$. Moreover, let $H(\cdot |\cdot)$ denote the relative entropy, with the convention that $H(Q|P)=\infty$ if $Q$ is not absolutely continuous with respect to $P$. We have
\begin{equation}\label{eq:entropy}
H(Q|P):=E^Q\left[\ln\left(\frac{dQ}{dP}\right)\right]=\frac{1}{2c\sigma^2}E^Q\left[\int_0^T c a_t^2 dt\right].
\end{equation}
Thus, we take the following as our definition of the single player's control problem:
\begin{equation}\label{eq:Vweak}
V(R,\tilde \mu):=\sup_{Q\in\mathcal{Q}} E^Q\left[R_{\tilde \mu}(X_T)\right]-2c\sigma^2 H(Q|P), \quad X_t=x_0+\sigma \omega_t.
\end{equation}

\begin{remark}
Here to keep notation simple, we define the filtration to be the one generated by the canonical process, but similar to Bayraktar et al.\ \cite[Remark 2.1]{BCZ18}, all results remain valid if we take $(\mathcal{F}_t)$ to be a larger filtration for which $\omega_t$ remains a Brownian motion.
\end{remark}


\subsection{Reduction via Schr\"odinger bridges}

Let  $X_t=x_0+\sigma \omega_t$ and $P=\bbW_0\circ X^{-1}$ as before; $P$ will serve as our reference measure.\footnote{In standard Schr\"odinger bridge problem, the reference measure $P$ is usually taken to be the stationary Wiener measure $\int \mathbb{W}_x dx$, but the disintegration argument works for any non-zero, non-negative, $\sigma$-finite measure on $\Omega$.} For any $Q\in\mathcal{P}(\Omega)$, write $Q_t$ for the time-$t$ marginal of $Q$, and $Q_{0,T}$ for the joint distribution of $Q$ at time $0$ and $T$.
Given a source distribution $\nu=\delta_{x_0}$ and a target distribution $\mu$, the Schr\"odinger bridge problem looks for an entropy-minimizing transport from $\nu$ to $\mu$:
\begin{equation}\label{dynamicSBP}
\underset{Q\in\mathcal{P}(\Omega)}{\text{minimize}}\quad H(Q|P) \quad \text{subject to}\quad Q_0=\nu, \quad Q_T=\mu,
\end{equation}
It is known, by a simple disintegration, that the solution to the Schr\"odinger bridge problem is given by (see F\"{o}llmer \cite{10.1007/BFb0086180} and also L\'{e}onard \cite{Leonard14}, Chen et al.\ \cite{CGP16})
\[Q^*(\cdot)=\int_{\mathbb{R}^2} P^{x,y}(\cdot) \pi^* (dx,dy),\]
where $P^{x,y}:=P(\cdot | X_0=x, X_T=y)$ is the law of a scaled Brownian bridge (scaled by $\sigma$), and $\pi^*$ is the solution to the following static optimization, assuming it exists:
\begin{equation}\label{staticSBP}
\underset{\pi\in\mathcal{P}(\mathbb{R}^2)}{\text{minimize}}\quad H(\pi|P_{0,T}) \quad \text{subject to}\quad \pi_0=\nu, \quad \pi_T=\mu.
\end{equation}
In addition, it holds that $H(Q^*|P)=H(\pi^*| P_{0,T})$. 
Since $\nu=P_0=\delta_{x_0}$, the static problem \eqref{staticSBP} is trivial, giving a minimum entropy of 
$H(\mu | \mathcal{N}(x_0, \sigma^2 T))$.

Going back to our control problem \eqref{eq:Vweak}, by splitting the optimization over $Q$ to a maximization over its time-$T$ marginal plus a constrained entropy minimization, we can utilize the 
equivalence between \eqref{dynamicSBP} and \eqref{staticSBP} and obtain
\begin{align*}
V(R,\tilde \mu)&= \sup_{\mu\in\mathcal{P}(\bbR): \,\mu\sim P_T}\sup_{Q\in\mathcal{Q}: \,Q_T=\mu} \int_\bbR R_{\tilde \mu}(x)d\mu(x)-2c\sigma^2 H(Q|P)\\
&\le  \sup_{\mu\in\mathcal{P}(\bbR): \,\mu\sim P_T}\sup_{Q\in\mathcal{P}(\Omega): Q_0=\nu, Q_T=\mu} \int_\bbR R_{\tilde \mu}(x)d\mu(x)-2c\sigma^2 H(Q|P)\\
&=\sup_{\mu\in\mathcal{P}(\bbR): \,\mu\sim P_T}\int_\bbR R_{\tilde \mu}(x)d\mu(x)-2c\sigma^2 \inf_{Q\in\mathcal{P}(\Omega): \,Q_0=\nu, \,Q_T=\mu}  H(Q|P)\\
&=\sup_{\mu\in\mathcal{P}(\bbR): \,\mu\sim P_T}\int_\bbR R_{\tilde \mu}(x)d\mu(x)-2c\sigma^2 \inf_{\pi\in\mathcal{P}(\bbR^2): \,\pi_0=\nu, \,\pi_T=\mu}  H(\pi|P_{0,T})\\
&=\sup_{\mu\in\mathcal{P}(\bbR): \,\mu\sim P_T} \int_\bbR R_{\tilde \mu}(x)d\mu(x)-2c\sigma^2 H(\mu| \mathcal{N}(x_0, \sigma^2 T)).
\end{align*}
Since $\int_{\mathbb{R}^2} P^{x_0,y}(\cdot) \mu (dy)\in\mathcal{Q}$ for any $\mu\sim P_T$, the inequality is in fact an equality. Thus,
\begin{equation}\label{eq:agent-problem2}
V(R,\tilde\mu)=\sup_{\mu\in\mathcal{P}(\bbR):\, \mu\sim P_T} \int_\bbR R_{\tilde \mu}(x)d\mu(x)-2c\sigma^2 H(\mu| \mathcal{N}(x_0, \sigma^2 T)).
\end{equation}

Let $\varphi$ be the standard normal probability density function (p.d.f.) and introduce
\begin{equation}\label{eq:f0}
f_0(x):=\frac{1}{\sigma\sqrt{T}}\varphi\left(\frac{x-x_0}{\sigma\sqrt{T}}\right).
\end{equation}
We finally arrive at a constrained calculus of variation problem over the p.d.f.\ of $\mu$:
\begin{equation}\label{eq:agent-problem2b}
V(R,\tilde \mu)=\sup_{f_{\mu}> 0: \,\int_{\bbR} f_{\mu}(x) dx=1}\int_\bbR \left\{R_{\tilde \mu}(x)  -2c\sigma^2 \ln\left(\frac{f_{\mu}(x)}{f_0(x)}\right)\right\} f_{\mu}(x)dx.
\end{equation}
which can be easily solved by the method of Lagrange multipliers.\footnote{Alternatively, one can directly drop the integral constraint, by observing that any non-negative function $f\not\equiv 0$ can be normalized to have integral equal to one.} The solution is provided below without proof. Once we find the optimal marginal $\mu^*$, we can recover $Q^*$ by 
$Q^*(\cdot)=\int_{\mathbb{R}^2} P^{x_0,y}(\cdot) \mu^* (dy).$

\begin{prop}\label{prop:best-response}
Given $R\in\mathcal{R}$ and $\tilde\mu \in\mathcal{P}(\bbR)$. Let 
\begin{equation*}
\beta(\tilde \mu):=\int_\bbR f_0(y)\exp\left(\frac{R_{\tilde \mu}(y)}{2c\sigma^2}\right) dy,
\end{equation*}
where $f_0$ is defined in \eqref{eq:f0}. Suppose $\beta(\tilde \mu)<\infty$. Then the optimal terminal distribution $\mu^*$ of the single player has p.d.f.
\begin{equation}\label{eq:fmu}
f_{\mu^*}(x)=\frac{1}{\beta(\tilde \mu)}f_0(x)\exp\left(\frac{R_{\tilde \mu}(x)}{2c\sigma^2}\right).
\end{equation}
The optimal value is given by $V(R,\tilde \mu)=2c\sigma^2 \ln \beta(\tilde \mu).$
\end{prop}

\begin{remark}
The Schr\"odinger bridge approach can also be adapted to the hitting time ranking game of Bayraktar et al.\ \cite{BCZ18}.
This calls for a variant of the Schr\"odinger bridge problem where the target distribution is not the time-$T$ marginal, but the law of first passage time of level zero. We detail this digression in the appendix for the interested readers.
\end{remark}

\begin{remark}
The model assumption that $X_T$ has a Gaussian $\mathbb{W}_0$-density is not essential in the best-response step; what matters is that the cost of effort can be written as a relative entropy $H(Q|P)$ where $P$ and $Q$ are the laws of the state process $X$ corresponding to zero effort and general effort, respectively. Suppose the $\mathbb{W}_0$-distribution of $X_T$ is $\mu_0$. Using the equivalence between the dynamic and static Schr\"odinger bridge problems, equations \eqref{eq:agent-problem2} and \eqref{eq:agent-problem2b} hold with $\mathcal{N}(x_0,\sigma^2 T)$ replaced by $\mu_0$ and $f_\mu$ replaced by $z=d\mu/d\mu_0$:
\begin{align*}
V(R,\tilde\mu)&=\sup_{\mu\in\mathcal{P}(\bbR):\, \mu\sim \mu_0} \int_\bbR R_{\tilde \mu}(x)d\mu(x)-2c\sigma^2 H(\mu|\mu_0)\\
&=\sup_{z(x)>0, \int z(x)d\mu_0(x)=1} \int_\bbR \left(R_{\tilde \mu}(x)-2c\sigma^2 \ln z(x) \right) z(x) d\mu_0(x).
\end{align*}
Using the method of Lagrange multipliers, one finds that the optimal $\mu^*$ is given by
\[\frac{d\mu^*}{d\mu_0}(x)=\frac{\exp\left(\frac{R_{\tilde \mu}(x)}{2c\sigma^2}\right)}{\int \exp\left(\frac{R_{\tilde \mu}(y)}{2c\sigma^2}\right) d\mu_0(y)},\]
which is similar to \eqref{eq:fmu}.
The derivation of the fixed point (see the proof of Theorem~\ref{thm:NE}), on the other hand, relies on the existence of a density, but not on the Gaussian property. 
\end{remark}

\subsection{Optimal effort}
The Schr\"odinger bridge approach allows us to compute the optimal target distribution easily, which is all we need to analyze equilibrium measures (see Section~\ref{sec:NE} for details). On the other hand, to get a more explicit description of the optimal effort, ideally as a feedback function $a^*(t,x)$ of time and state, we still need to go back to the dynamic control formulation of the Schr\"odinger bridge problem. We can utilize some existing results in, for example, Chen et al.\ \cite{CGP16}.


Recall that under the reference measure $P$, the canonical process is a scaled Brownian motion with transition density
\[p(t,x, s,y):=\frac{1}{\sigma\sqrt{s-t}}\varphi\left(\frac{y-x}{\sigma\sqrt{s-t}}\right).\]
Note that $f_0(y)=p(0,x_0,T,y)$. Define
\[\psi(t,x):=\int_\bbR p(t,x,T,y) \frac{f_{\mu^*}(y)}{f_0(y)}dy, \quad \hat \psi(t,x)=p(0,x_0,t,x).\]
It can be easily checked that $\psi, \hat \psi$ satisfy $\psi(0,x)\hat\psi(0,x)=\delta_{x_0}(x)$, $\psi(T,x)\hat\psi(T,x)=f_{\mu^*}(x)$,
\[\psi(t,x)=\int_\bbR p(t,x,T,y) \psi(T,y) dy, \quad \text{and}\quad \hat \psi(t,x)=\int_\bbR p(0,y,t,x) \hat \psi(0,y) dy.\]
By Chen et al.\ \cite[p.\ 679-680]{CGP16}, the optimal coupling $Q^*$ has Markovian drift $a^*$ given by
\[a^*(t,x)=\sigma^2 \partial_x \ln \psi(t,x).\footnote{Chen et al.\ \cite{CGP16} assumes $\sigma=1$, but the proof can be easily generalized to nonzero $\sigma$.}\]
Using \eqref{eq:fmu}, we obtain
\[\psi(t,x)=\frac{1}{\beta(\tilde \mu)} E\left[\exp\left(\frac{R_{\tilde \mu}(x+\sigma\sqrt{T-t}Z)}{2c\sigma^2}\right)\right], \quad Z\sim \mathcal{N}(0,1).\]
Comparing with Bayraktar and Zhang \cite[eq.\ (3.3)]{BZ16}, we see that $u(t,x):=\beta(\tilde \mu) \psi(t,x)$ is precisely the Cole-Hopf transformation of the value function of the original control problem \eqref{eq:original-control-prob}. Replacing $\psi$ by $u$, we recover the same optimal Markovian control as Bayraktar and Zhang \cite{BZ16}:
\begin{equation}\label{eq:astar}
a^*(t,x)=\sigma^2 \frac{u_x(t,x)}{u(t,x)}.
\end{equation}
When $R$ is bounded, it is shown in Bayraktar and Zhang \cite{BZ16} that $\lim_{x\rightarrow \pm \infty} a^*(t,x)=0$, meaning players show slackness when having a very big lead, and give up when falling far behind.

\begin{remark}
For bounded rewards, Bayraktar and Zhang \cite{BZ16} also showed that the controlled diffusion $dX_t=a^*(t,X_t)dt+\sigma dB_t$ in fact has a unique strong solution. From there,  one can mimic the change of measure technique in Bayraktar et al.\ \cite{BCZ18} to obtain the optimal terminal distribution 
\eqref{eq:fmu}. 
An advantage of the weak formulation, beside the connection to optimal transport theory, is that it avoids the hassle of having to verify the regularity of $a^*$ near $x=0$.
\end{remark}

\section{Characterization of equilibrium}\label{sec:NE}

We say $\mu\in\mathcal{P}(\bbR)$ is an \emph{equilibrium (terminal distribution)} if it is a fixed point of the best response mapping: $\tilde \mu \mapsto  Q_T$, where $Q\in\mathcal{Q}$ is the optimal control for $V(R,\tilde\mu)$. By \eqref{eq:fmu}, we have the following characterization for general rewards functions.

\begin{thm}\label{thm:NE0}
Let $R\in\mathcal{R}$ and $\mu\in\mathcal{P}(\bbR)$ satisfy 
\[\beta(\mu)=\int_\bbR f_0(y)\exp\left(\frac{R_{\mu}(y)}{2c\sigma^2}\right) dy<\infty.\]
(The above condition always holds when $R\in\mathcal{R}_b$.)
Then $\mu$ is an equilibrium if and only if it has a strictly positive density satisfying
\begin{equation}\label{eq:fixed-point}
f_{\mu}(x)=\frac{1}{\beta(\mu)}f_0(x)\exp\left(\frac{R_{\mu}(x)}{2c\sigma^2}\right).
\end{equation}
The associated game value is given by $V(R,\mu)=2c\sigma^2 \ln \beta(\mu)$.
\end{thm}

Specializing to the subclass of reward functions
\[\mathcal{R}_b^{rm}:=\{R\in\mathcal{R}_b: \text{ $R(x,r,m)$  is independent of $x$ and continuous in $m$}\},\]
we obtain a semi-explicit characterization.

\begin{thm}\label{thm:NE}
Suppose $R\in\mathcal{R}_b^{rm}$. Then there exists at least one equilibrium. $\mu\in\mathcal{P}(\bbR)$ is an equilibrium terminal distribution of the project value if and only if its quantile function $q_\mu$ satisfies
\begin{equation}\label{eq:q_mu}
q_\mu(r)=x_0+\sigma \sqrt{T} N^{-1}\left(\frac{\int_0^r \exp\left(-\frac{R(z,m_\mu)}{2c\sigma^2}\right)dz}{\int_0^1 \exp\left(-\frac{R(z,m_\mu)}{2c\sigma^2}\right)dz}\right),
\end{equation}
where $N(\cdot)$ is the standard normal c.d.f.\ and $m_\mu=\int_{-\infty}^\infty y d\mu(y)$ is a solution of
\begin{equation}\label{eq:m}
m=x_0+\sigma\sqrt{T}\int_0^1 N^{-1}\left(\frac{\int_0^r \exp\left(-\frac{R(z,m)}{2c\sigma^2}\right)dz}{\int_0^1 \exp\left(-\frac{R(z,m)}{2c\sigma^2}\right)dz}\right)dr.
\end{equation}
The associated game value is given by
\begin{equation*}\label{eq:game-value}
V(R,\mu)=2c\sigma^2 \ln \beta(\mu)=-2c\sigma^2 \ln\left(\int_0^1 \exp\left(-\frac{R(z,m_\mu)}{2c\sigma^2}\right) dz\right).
\end{equation*}
\end{thm}

\begin{proof}
Since $R$ is bounded, we only need to look for solutions of the fixed point equation \eqref{eq:fixed-point}. 
Let $y(\cdot)$ be the c.d.f.\ of the random variable $F_\mu(\mathcal{N}(x_0, \sigma^2 \sqrt{T}))$, i.e.
\[y(r)=N\left(\frac{q_\mu(r)-x_0}{\sigma\sqrt{T}}\right).\]
Since any fixed point $\mu$ has a positive density, we can differentiate $y(r)$ and use \eqref{eq:fixed-point} to get
\begin{align*}
y'(r)&=\varphi\left(\frac{q_\mu(r)-x_0}{\sigma\sqrt{T}}\right)\frac{1}{\sigma\sqrt{T}}\frac{1}{f_\mu(q_\mu(r))}=\frac{f_0(q_\mu(r))}{f_\mu(q_\mu(r))}\\
&=\beta(\mu)\exp\left(-\frac{R_\mu(q_\mu(r))}{2c\sigma^2}\right)=\beta(\mu)\exp\left(-\frac{R(r,m_\mu)}{2c\sigma^2}\right).
\end{align*}
Using $y(0)=0$ and $y(1)=1$, we find that
\[y(r)=\beta(\mu)\int_0^r \exp\left(-\frac{R(z,m_\mu)}{2c\sigma^2}\right)dz,\]
and
\[\beta(\mu)=\left(\int_0^1 \exp\left(-\frac{R(z,m_\mu)}{2c\sigma^2}\right)dz\right)^{-1}.\]
It follows that
\[N\left(\frac{q_\mu(r)-x_0}{\sigma\sqrt{T}}\right)=y(r)=\frac{\int_0^r \exp\left(-\frac{R(z,m_\mu)}{2c\sigma^2}\right)dz}{\int_0^1 \exp\left(-\frac{R(z,m_\mu)}{2c\sigma^2}\right)dz},\]
from which we get \eqref{eq:q_mu}. To determine $m_\mu$, we integrate \eqref{eq:q_mu} from $r=0$ to $r=1$ and use that $m_\mu=\int_0^1 q_\mu(r)dr$. This leads to equation \eqref{eq:m}. It remains to show that \eqref{eq:m}
has a solution.

Let $g(m)$ be the right hand side of \eqref{eq:m}. We want to show $g$ has a fixed point. Since $R$ is bounded, it can be shown that $C^{-1}\le y'(r)\le C$ where
\[ C=\exp\left(\frac{\| R(1,m)-R(0,m)\|_{\infty}}{2c\sigma^2}\right).\]
It follows that
\begin{align*}
\left|\frac{g(m)-x_0}{\sigma\sqrt{T}}\right|&=\left|\int_0^1 N^{-1}(y(r))dr\right|=\left|\int_0^1 N^{-1}(y) \frac{dy}{y'(r)}\right|\le C \int_0^1 \left| N^{-1}(y)\right| dy=C\sqrt{\frac{2}{\pi}}.
\end{align*}
So the range of $g$ is contained in a compact interval. Moreover, $g$ is continuous on this interval since $R$ is assumed to be continuous in $m$. By Brouwer's fixed point theorem, $g$ has a fixed point. 
\end{proof}

\begin{remark}
Observe that the equilibrium distribution $\mu$ does not change if we add any bounded function $\kappa(m)$ to the reward. In other words, any bounded compensation that is solely based on the mean performance of the population does not really incentivize the players.
\end{remark}

\begin{remark}\label{rmk:effort}
When $R\in \mathcal{R}^{rm}_b$ is further independent of $m$ (i.e.\ purely rank-based), the equilibrium is unique. In this case, the total effort of the population (or the expected cumulative effort of a representative player) is given by
\[A(R):=E\int_0^T a^*(t,X^*_t)dt=m_\mu-x_0=\sigma\sqrt{T}\int_0^1 N^{-1}\left(\frac{\int_0^r \exp\left(-\frac{R(z)}{2c\sigma^2}\right)dz}{\int_0^1 \exp\left(-\frac{R(z)}{2c\sigma^2}\right)dz}\right)dr.\]
\end{remark}

\begin{remark}\label{rmk:subclassNE}
If we confine ourselves to the subclass of equilibria which satisfy $\beta(\mu)<\infty$, 
then all results in this section can be restated with $R\in \mathcal{R}^{rm}$ which is obtained from $\mathcal{R}^{rm}_b$ by dropping the boundedness requirement.
\end{remark}

In the next two sections, we focus on bounded rewards that are purely rank-based:
\[\mathcal{R}_b^r:=\{R\in\mathcal{R}_b: R(x,r,m) \text{ is independent of $x$ and $m$}\}.\]
Each of these rewards induces a unique equilibrium, which facilitates the study of comparative statics and optimal reward design. In this case, we write $\mathcal{V}(R)$ for the unique game value.

\section{Effect of reward inequality}\label{sec:Strack}
\begin{defn}
Given two reward functions $R, \tilde R\in\mathcal{R}_b^r$, we say $R$ is more unequal than $\tilde R$ in Lorenz order (or $R$ majorizes $\tilde R$), written as $R\succ \tilde R$, if $\int_0^1 R(r)dr=\int_0^1 \tilde R(r)dr$ and
\[\int_0^z R(r)dr \le \int_0^z \tilde R(r)dr \quad \forall\, z\in [0,1].\]
\end{defn}

\begin{thm}
Suppose $R, \tilde R\in\mathcal{R}_b^r$ and $R\succ \tilde R$, then the associated game values satisfy $\mathcal{V}(R)\le \mathcal{V}(\tilde R)$; that is, reward inequality decreases the game value.
\end{thm}
\begin{proof}
First assume $R, \tilde R\in \mathcal{R}^r_n$, where $\mathcal{R}^r_n$ is the set of piecewise constant reward functions of the form
\[R(r)=\sum_{i=1}^n R_i 1_{\{(i-1)/n\le r< i/n\}}+R_n 1_{\{r=1\}}.\]
In this case, the Lorenz order translates to $\sum_{i=1}^n R_i=\sum_{i=1}^n \tilde R_i$ and $\sum_{i=1}^k R_i\le \sum_{i=1}^k \tilde R_i$ for all $k\in\{1,\ldots, n\}$; that is, $(R_1, \ldots, R_n)$ majorizes $(\tilde R_1, \ldots, \tilde R_n)$. 
By Marshall et al.\ \cite[Proposition 4.B.1]{MajorizationTheory}, $\tilde R\prec R$ if and only if $\sum_{i=1}^n g(\tilde R_i)\le \sum_{i=1}^n g(R_i)$ for all continuous convex functions $g$. Take $g(x)=\exp(-\frac{x}{2c\sigma^2})$, we obtain
\[\int_0^1 \exp\left(-\frac{\tilde R(r)}{2c\sigma^2}\right) dr=\frac{1}{n}\sum_{i=1}^n g(\tilde R_i)\le \frac{1}{n}\sum_{i=1}^n g(R_i)= \int_0^1 \exp\left(-\frac{R(r)}{2c\sigma^2}\right) dr,\]
which is equivalent to $\mathcal{V}(\tilde R)\ge \mathcal{V}(R)$.
This finishes the proof for piecewise constant reward functions.

For general $R, \tilde R \in\mathcal{R}_b^r$, we approximate $\mathcal{V}(R)$ and $\mathcal{V}(\tilde R)$ by the Riemann sums $\mathcal{V}(R^{(n)})$ and $\mathcal{V}(\tilde R^{(n)})$, respectively, where $R^{(n)}, \tilde R^{(n)}\in \mathcal{R}^r_n$. Moreover, by the mean value theorem, $R^{(n)}, \tilde R^{(n)}$ can always be chosen to satisfy $\sum_{i=1}^k R^{(n)}_i=\int_0^{k/n} R(r)dr$ and $\sum_{i=1}^k \tilde R^{(n)}_i=\int_0^{k/n} \tilde R(r)dr$ for all $k\in\{1, \ldots, n\}$. This ensures that the discretization preserves the Lorenz order. The result then follows from the previous step and passing to the limit.
\end{proof}

\begin{remark}
The maximum game value is attained by the most equal reward function, namely, the uniform reward. This can also be directly seen from Jenssen's inequality:
\[\int_0^1 \exp\left(-\frac{R(r)}{2c\sigma^2}\right)dr\ge \exp\left(-\frac{\int_0^1 R(r)dr}{2c\sigma^2}\right),\]
with equality attained if and only if $R$ is constant. From another perspective, the expected reward in equilibrium is always equal to $\int_0^1 R(r)dr$ by symmetry, while the expected cost of effort is minimized to zero under the uniform reward, when nobody exerts any effort. Since uniform reward induces zero effort,  the expected total effort clearly does not have the same monotonicity as the game value with respect to reward inequality (cf. Section~\ref{sec:opt-effort}).

\end{remark}

\section{Tournament design}\label{sec:tourn}

Denote the mapping from $R\in\mathcal{R}_b^r$ to the unique equilibrium $\mu$ by 
\[\mathcal{E}: \mathcal{R}_b^r\rightarrow {\mathcal{P}(\bbR)}.\] 
From \eqref{eq:q_mu}, we see that $\mathcal{E}$ is translation invariant, i.e. $\mathcal{E}(R+C)=\mathcal{E}(R)$ for any constant $C$. Let $\mathcal{P}^+(\bbR)$ be the set of probability measures on $\bbR$ that have strictly positive density. For $\mu\in \mathcal{P}^+(\bbR)$, define the normalized density
\[\zeta_\mu:=f_\mu/f_0.\]

\subsection{Realizing a target equilibrium distribution}

Suppose the principal has in mind a target distribution $\mu$ of the terminal project value in equilibrium. He wants to know whether that is feasible via a purely rank-based reward, and if yes, how should he design the reward to achieve it? The following theorem completely characterizes the set of feasible equilibria and the reward functions that induce them.

\begin{thm}\label{thm:reverse_eng} \
\begin{itemize}
\item[(i)] The set of equilibria attainable by a purely rank-based reward is given by
\[\mathcal{E}(\mathcal{R}_b^r)=\{\mu\in \mathcal{P}^+(\bbR): \zeta_\mu, 1/\zeta_\mu \text{ are bounded and $\zeta_\mu$ is increasing}\}.\]
\item[(ii)] If $\mu\in \mathcal{E}(\mathcal{R}_b^r)$, then 
\[\mathcal{E}^{-1}(\mu)=\{2c\sigma^2\ln \zeta_\mu(q_\mu(r))+C: C\in\bbR\}.\]
\item[(iii)] Suppose we impose additional reservation ``utility" constraint $\mathcal{V}(R)\ge V_0$ and budget constraint $\int_0^1 R(r)dr\le K$, then the constant $C$ in (ii) is restricted to
\[V_0\le C\le K-2c\sigma^2 H\left(\mu | \mathcal{N}(x_0, \sigma^2 T)\right),\]
where 
\[H\left(\mu | \mathcal{N}(x_0, \sigma^2 T)\right)=\int_{-\infty}^\infty \ln \zeta_\mu(y) d\mu(y)=\int_0^1 \ln \zeta_\mu(q_\mu(r))dr.\] 
In particular, such a $C$ exists if and only if
\begin{equation*}
H\left(\mu | \mathcal{N}(x_0, \sigma^2 T)\right)\le \frac{K-V_0}{2c\sigma^2}.
\end{equation*}
\end{itemize}
\end{thm}

\begin{proof}
(i) From Theorem~\ref{thm:NE0}, 
we know that the normalized density $\zeta_\mu$ of any equilibrium $\mu$ is increasing and log-bounded. Conversely, given any $\mu\in\mathcal{P}^+(\bbR)$ with such properties, it is easy to check that $\mu$ satisfies \eqref{eq:q_mu} with purely rank-based reward function $R_0(r)=2c\sigma^2\ln \zeta_\mu(q_\mu(r))$:
\[\frac{\int_0^r \exp\left(-\frac{2c\sigma^2\ln \zeta_\mu(q_\mu(z))}{2c\sigma^2}\right)dz}{\int_0^1 \exp\left(-\frac{2c\sigma^2\ln \zeta_\mu(q_\mu(z))}{2c\sigma^2}\right)dz}
=\frac{\int_0^r \left(\zeta_\mu(q_\mu(z))\right)^{-1}dz}{\int_0^1 \left(\zeta_\mu(q_\mu(z))\right)^{-1}dz}=\frac{\int_{-\infty}^{q_\mu(r)}\frac{1}{\zeta_\mu(y)}d\mu(y)}{\int_{-\infty}^{\infty}\frac{1}{\zeta_\mu(y)}d\mu(y)}=N\left(\frac{q_\mu(r)-x_0}{\sigma\sqrt{T}}\right).\]

(ii) If $R(r)$ is another function in $\mathcal{R}_b^r$ that attains $\mu$ in equilibrium, then
\[N\left(\frac{y-x_0}{\sigma\sqrt{T}}\right)=\frac{\int_0^{F_\mu(y)} \exp\left(-\frac{R(z)}{2c\sigma^2}\right)dz}{\int_0^1 \exp\left(-\frac{R(z)}{2c\sigma^2}\right)dz}\]
by \eqref{eq:q_mu}. Differentiating both sides with respect to $y$ and setting $y=q_\mu(r)$, we obtain
\[\int_0^1 \exp\left(-\frac{R(z)}{2c\sigma^2}\right)dz= \zeta_\mu(q_\mu(r))\exp\left(-\frac{R(r)}{2c\sigma^2}\right)=\exp\left(\frac{R_0(r)-R(r)}{2c\sigma^2}\right).\]
Since the left hand side is independent of $r$, $R-R_0$ must be constant.

(iii) Let $R(r)=R_0(r)+C$ be a reward function realizing $\mu$ in equilibrium. By Theorem~\ref{thm:NE}, the game value $\mathcal{V}(R)=\mathcal{V}(R_0)+C=C$. Hence $\mathcal{V}(R)\ge V_0$ if and only if $C\ge V_0$. We also have
\[\int_0^1 R(r)dr=C+\int_0^1 2c\sigma^2\ln \zeta_\mu(q_\mu(r)) dr=C+2c\sigma^2 H\left(\mu | \mathcal{N}(x_0, \sigma^2 T)\right).\]
So $\int_0^1 R(r)dr\le K$ if and only if $C\le K-2c\sigma^2 H\left(\mu | \mathcal{N}(x_0, \sigma^2 T)\right)$.
\end{proof}

Theorem~\ref{thm:reverse_eng} allows us to convert many optimal reward design problems into problems about finding the optimal target equilibrium distribution. We gave three solvable examples below.

\subsection{Maximizing rank-$\alpha$ performance}

Fix a number $\alpha\in(0,1)$, a reservation utility $V_0$ and a budget $K\ge V_0$. We look for a reward function $R\in\mathcal{R}^r_b$ which meets both the reservation utility requirement and the budget constraint, and which maximizes the $\alpha$-quantile of $\mathcal{E}(R)$. 
Define the set of feasible reward functions by 
\[\mathcal{H}:=\left\{R\in\mathcal{R}_b^r: \int_0^1 R(r)dr \le K  \text{ and } \mathcal{V}(R)=V(R,\mathcal{E}(R))\ge V_0\right\}.\]
The optimization problem reads
\[Q(\alpha):=\sup_{R\in\mathcal{H}} q_{\mathcal{E}(R)}(\alpha).\]

\begin{thm}
The optimal quantile $Q(\alpha)$ is uniquely attained (up to a.e.\ equivalence) by the step function
\[R^*(r)=V_0+\begin{cases}
2c\sigma^2 \ln \frac{\alpha}{x_{\alpha}}, & 0\le r<\alpha,\\
2c\sigma^2 \ln \frac{1-\alpha}{1-x_{\alpha}}, & \alpha \le r\le 1,
\end{cases}\]
where $x_\alpha$ is the unique solution in $[\alpha,1)$ to the equation
\begin{equation*}\label{eq:x_alpha}
\alpha \ln \frac{\alpha}{x}+(1-\alpha)\ln \frac{1-\alpha}{1-x}=\frac{K-V_0}{2c\sigma^2}.
\end{equation*}
Let $\mu^*=\mathcal{E}(R^*)$ and $f_0$ be given by \eqref{eq:f0}. We have
\[Q(\alpha)=q_{\mu^*}(\alpha)=x_0+\sigma\sqrt{T} N^{-1}\left(x_\alpha\right),\] 
and
\[f_{\mu^*}(x)=\begin{cases}
\frac{\alpha}{x_\alpha}f_0(x), & x< Q(\alpha),\\
\frac{1-\alpha}{1-x_\alpha}f_0(x), & x > Q(\alpha).
\end{cases}\]
\end{thm}

\begin{proof}
By Theorem~\ref{thm:reverse_eng}, $\mu\in\mathcal{E}(\mathcal{H})$ if and only if $\mu\in\mathcal{E}(\mathcal{R}_b^r)$ and 
\begin{equation}\label{eq:constraint}
H\left(\mu | \mathcal{N}(x_0, \sigma^2 T)\right)=\int_0^1 \ln \zeta_\mu(q_\mu(r))dr\le \frac{K-V_0}{2c\sigma^2}.
\end{equation}
So we can equivalently formulate the optimization problem as one having $\mu$ as the decision variable, $q_{\mu}(\alpha)$ as the objective function,
and $\mu\in\mathcal{E}(\mathcal{R}_b^r)$ and \eqref{eq:constraint} as the constraints. 

Maximizing $q_\mu(\alpha)$ is equivalent to maximizing 
\begin{align*}
N\left(\frac{q_\mu(\alpha)-x_0}{\sigma\sqrt{T}}\right)&=\int_{-\infty}^{q_\mu(\alpha)}\frac{1}{\sigma\sqrt{T}}\varphi\left(\frac{y-x_0}{\sigma\sqrt{T}}\right)dy\\
&=\int_{0}^{\alpha}\frac{1}{\sigma\sqrt{T}}\varphi\left(\frac{q_\mu(r)-x_0}{\sigma\sqrt{T}}\right)\frac{dr}{f_\mu(q_\mu(r))}=\int_0^\alpha \frac{dr}{\zeta_\mu(q_\mu(r))}.
\end{align*}
For any feasible equilibrium distribution $\mu$, let $h:=1/(\zeta_\mu \circ q_\mu)$ which implies $\zeta_\mu=1/(h\circ F_\mu)$, $\int_0^1 h(r)dr=1$, and $\mu=\mathcal{E}(-2c\sigma^2 \ln h)$. In particular, the mapping from $\mu$ to $h$ is one-to-one. 
Further rewrite the optimization problem as
\begin{equation}\label{prob-h}
\text{maximize} \quad \mathcal{J}(h)=\int_0^\alpha h(r)dr \quad \text{s.t.}\quad \int_0^1 -\ln h(r)dr \le \frac{K-V_0}{2c\sigma^2} \ \text{ and } \int_0^1 h(r)dr=1,
\end{equation}
where $h$ is also constrained to be positive, decreasing, bounded and bounded away from zero, as translated from $\mu\in\mathcal{E}(\mathcal{R}_b^r)$.
Each feasible $\mu$ clearly induces a feasible $h$. Conversely, for any feasible $h$, define $\mu=\mathcal{E}(-2c\sigma^2 \ln h)$. Then $-2c\sigma^2 \ln h(r)=2c\sigma^2 \ln \zeta_\mu (q_\mu(r))+C$ for some constant $C$ by Theorem~\ref{thm:reverse_eng}. Together with the constraints in \eqref{prob-h}, we find that $h=1/(\zeta_{\mu} \circ q_{\mu})$ and that $\mu$ is feasible. Thus, the mapping from feasible $\mu$ to feasible $h$ is in fact bijective, which implies that it suffices for us to solve problem \eqref{prob-h}. Any optimal $h$ induces an optimal $\mu=\mathcal{E}(-2c\sigma^2 \ln h)$ which can be realized by the reward function
\[R=V_0-2c\sigma^2 \ln h.\]
Here we have added the constant $V_0$ to $-2c\sigma^2 \ln h$ to ensure that $R\in\mathcal{H}$. 
The rest of the proof is devoted to solving the equivalent problem \eqref{prob-h}. 

We first show that the constant $x_\alpha$ given in the theorem statement is well-defined. Let
\[g(x):=\alpha \ln \frac{\alpha}{x}+(1-\alpha)\ln \frac{1-\alpha}{1-x}, \quad 0<x<1.\]
It can be shown that $g(x)$ is strictly decreasing on $(0,\alpha)$ and strictly increasing on $(\alpha,1)$, hence has a global minimum at $x=\alpha$ with $g(\alpha)=0$. Moreover, $g(x)\rightarrow \infty$ as $x\rightarrow 0$ or $1$. Since $K\ge V_0$, by intermediate value theorem, the equation
\[g(x)=\frac{K-V_0}{2c\sigma^2}\]
has at least one solution.
When $K=V_0$, $x=\alpha$ is the unique solution. When $K>V_0$, there are two solutions: one in $(0,\alpha)$ and the other in $(\alpha,1)$. In both cases. $x_\alpha\in [\alpha,1)$ is well-defined.

Next, we show that
\[h^*(r):=\begin{cases}
\frac{x_\alpha}{\alpha}, & 0\le r< \alpha,\\
\frac{1-x_\alpha}{1-\alpha}, & \alpha\le r\le 1.
\end{cases}\]
is the unique optimizer of problem \eqref{prob-h}.
Since $0<\alpha\le x_\alpha<1$, it is clear that $h^*$ is decreasing, bounded and bounded away from zero. Straightforward calculation also shows that
\[\int_0^1 -\ln h^*(r)dr
=g(x_\alpha)=\frac{K-V_0}{2c\sigma^2} \quad \text{and} \quad \int_0^1 h^*(r)dr=1.\]
Therefore, $h^*$ satisfies all the feasibility constraints. Given any other feasible $h$. We have, by repeated application of Jensen's inequality, that
\begin{align*}
-\ln \left(\frac{1}{\alpha}\int_0^\alpha h(r)dr\right)&\le \frac{1}{\alpha} \int_0^\alpha -\ln h(r)dr = \frac{1}{\alpha}\left(\int_0^1 -\ln h(r)dr-\int_\alpha^1 -\ln h(r)dr\right)\\
&\le  \frac{1}{\alpha}\left(\frac{K-V_0}{2c\sigma^2}-\int_\alpha^1 -\ln h(r)dr\right)=\frac{K-V_0}{2\alpha c\sigma^2}+\frac{1-\alpha}{\alpha}\frac{1}{1-\alpha}\int_\alpha^1 \ln h(r)dr\\
&\le \frac{K-V_0}{2\alpha c\sigma^2}+\frac{1-\alpha}{\alpha}\ln \left(\frac{1}{1-\alpha}\int_\alpha^1 h(r)dr\right)\\
&= \frac{K-V_0}{2\alpha c\sigma^2}+\frac{1-\alpha}{\alpha}\ln \left(\frac{1}{1-\alpha}\left(1-\int_0^\alpha h(r)dr\right)\right).
\end{align*}
That is
\[g(\mathcal{J}(h))\le  \frac{K-V_0}{2 c\sigma^2}=g(x_\alpha).\]
We claim that $\mathcal{J}(h)\le x_\alpha=\mathcal{J}(h^*)$. Suppose on the contrary that $\mathcal{J}(h)> x_\alpha$. Then since $x_\alpha\ge \alpha$ and $g$ is strictly increasing on $(\alpha,1)$, we must have $g(\mathcal{J}(h))>g(x_\alpha)$, which is a contradiction. Thus, we have proved that $h^*$ is optimal. In fact, $h^*$ is the unique optimizer, since $\mathcal{J}(h)=\mathcal{J}(h^*)$ would imply all Jensen's inequalities above are equalities. This holds if and only if $h$ is constant on $[0,\alpha)$ and $(\alpha,1]$. We then use $\mathcal{J}(h)=\mathcal{J}(h^*)=x_\alpha$ and $\int_0^1 h(r)dr=1$ to deduce that $h=h^*$.

Finally, we argue that the optimal reward function $R^*=V_0-2c\sigma^2 \ln h^*$ induced by $h^*$ is also unique. Because of the bijection between $\mu$ and $h$, we know that $\mu^*=\mathcal{E}(-2c\sigma^2 \ln h^*)$ is the unique optimal equilibrium distribution. Note that $H\left(\mu^* | \mathcal{N}(x_0, \sigma^2 T)\right)=\int_0^1 -\ln h^*(r)dr=\frac{K-V_0}{2c\sigma^2}$. By Theorem~\ref{thm:reverse_eng}, 
\[\mathcal{E}^{-1}(\mu^*)\cap \mathcal{H}=\{C-2c\sigma^2 \ln h^*: V_0\le C\le K-2c\sigma^2 H\left(\mu^* | \mathcal{N}(x_0, \sigma^2 T)\right)\}=\{R^*\}.\]
The remaining theorem statements follow from direct calculation.
\end{proof}

\begin{remark}
One can also replace the reservation utility constraint by the hard constraint: $R\ge R_0$. Similar to Bayraktar et al.\ \cite[Theorem 6.2]{BCZ18}, the optimal reward function in this case is the equal reward with cutoff rank $\alpha$, i.e.\ $R(r)=R_0+\frac{K-R_0}{1-\alpha}1_{[\alpha,1]}(r)$.
\end{remark}

\subsection{Maximizing net profit}

Suppose each terminal output $y$ generates a profit $g(y)$ for the principal, where $g$ is a bounded increasing function. The goal is to find $R\in\mathcal{R}_b^r$ such that $\mathcal{V}(R)\ge V_0$ and the net profit
\[\int_{-\infty}^\infty g(y)d\mathcal{E}(R)(y)-\int_0^1 R(r)dr\] 
is maximized.

\begin{thm} The optimal net profit is given by
\[\frac{\int_{-\infty}^\infty \exp\left(\frac{g(y)}{2c\sigma^2}\right) f_0(y)g(y) dy}{\int_{-\infty}^\infty \exp\left(\frac{g(y)}{2c\sigma^2}\right) f_0(y)dy},\]
and is uniquely attained by 
\[R^*(r)=V_0+g(q_{\mu^*}(r))-2c\sigma^2 \ln \left(\int_{-\infty}^\infty \exp\left(\frac{g(z)}{2c\sigma^2}\right) f_0(z)dz\right),\]
where $f_0$ is given by \eqref{eq:f0}, and
\[f_{\mu^*}(y)=\frac{\exp\left(\frac{g(y)}{2c\sigma^2}\right) f_0(y)}{\int_{-\infty}^\infty \exp\left(\frac{g(z)}{2c\sigma^2}\right) f_0(z)dz}.\]
\end{thm}

\begin{proof}
By Theorem~\ref{thm:reverse_eng}, it suffices for us to look for the optimal $\mu\in\mathcal{E}(\mathcal{R}_b^r)$ which can then be realized by $R(r)=2c\sigma^2\ln \zeta_\mu(q_\mu(r))+C$ for any $C\ge V_0$. It is clear that the principal should pick $C=V_0$ to minimize the cost. Write $R^\mu(r)=2c\sigma^2\ln \zeta_\mu(q_\mu(r))+V_0$. We then have
\[\int_0^1 R^\mu(r)dr=V_0+2c\sigma^2 \int_{-\infty}^{\infty} \ln \left(\frac{f_\mu(y)}{f_0(y)}\right) f_\mu(y) dy.\]
The optimization problem over $\mu$ is given by
\begin{equation}\label{prob-profit}
\begin{aligned}
\text{maximize} \quad &\int_{-\infty}^\infty g(y) f_\mu(y)dy-V_0-2c\sigma^2 \int_{-\infty}^{\infty} \ln \left(\frac{f_\mu(y)}{f_0(y)}\right) f_\mu(y) dy\\
 \text{s.t.} \quad & \int_{-\infty}^\infty f_\mu(y)dy= 1, \quad \ln \left(\frac{f_\mu}{f_0}\right) \text{ is bounded and increasing}.
\end{aligned}
\end{equation}
To solve problem \eqref{prob-profit}, we define 
\begin{align*}
L(f_\mu,\lambda)&=\int_{-\infty}^\infty g(y) f_\mu(y)dy-2c\sigma^2 \int_{-\infty}^{\infty} \ln \left(\frac{f_\mu(y)}{f_0(y)}\right) f_\mu(y) dy+\lambda\left(1-\int_{-\infty}^\infty f_\mu(y)dy\right).
\end{align*}
For each fixed $\lambda\in \bbR$, the integrand above attains its pointwise maximum at
\[f_\mu(y)=\exp\left(\frac{g(y)-\lambda}{2c\sigma^2}-1\right)f_0(y).\]
Clearly, since $g$ is bounded and increasing, so is $\ln (f_\mu/f_0)$.
We then find $\lambda$ by
\begin{equation*}
1=\int_{-\infty}^\infty f_\mu(y)dy=
\int_{-\infty}^\infty \exp\left(\frac{g(y)-\lambda}{2c\sigma^2}-1\right)f_0(y) dy,
\end{equation*}
giving
\[\exp\left(1+\frac{\lambda}{2c\sigma^2}\right)=\int_{-\infty}^\infty \exp\left(\frac{g(y)}{2c\sigma^2}\right) f_0(y)dy.\]
The formulas for $f_{\mu^*}$ and $R^*=R^{\mu^*}$ then follow.
\end{proof}

\subsection{Maximizing total effort}\label{sec:opt-effort}

Let $K\ge V_0$ be given. We look for a purely rank-based reward function $R$  which maximizes the total effort
\[A(R)=m_{\mathcal{E}(R)}-x_0=\int_{-\infty}^\infty y d{\mathcal{E}(R)}(y)-x_0,\]
subject to the reservation utility constraint $\mathcal{V}(R)\ge V_0$ and the budget constraint $\int_0^1 R(r)dr\le K$.
\begin{thm}\label{thm:opt-effort}
$\sup_{R\in\mathcal{R}_b^r}A(R)=\sqrt{(K-V_0)T/c}$. When $K=V_0$, the unique optimal reward is given by $R^*(r)\equiv V_0$. When $K>V_0$, as $M\rightarrow \infty$, an $O(e^{-M/\lambda})$-optimal reward is
\[R_M(r)=2c\sigma^2\ln \zeta_{\mu_M}(q_{\mu_M}(r))+V_0,\]
where
\[f_{\mu_M}(y)=\frac{f_0(y)\exp\left(\frac{y\wedge M\vee(-M)}{\lambda}\right)}{\int_{-\infty}^{\infty} f_0(z)\exp\left(\frac{z\wedge M\vee(-M)}{\lambda}\right)dz}, \quad \lambda=\sigma^2 \sqrt{\frac{cT}{K-V_0}}.\]
\end{thm}
\begin{proof}
By Theorem~\ref{thm:reverse_eng}, it suffices for us to look for an optimal target distribution $\mu$ satisfying 
\[H\left(\mu | \mathcal{N}(x_0, \sigma^2 T)\right)\le \frac{K-V_0}{2c\sigma^2}.\]
Such a $\mu$, if lies in $\mathcal{E}(\mathcal{R}_b^r)$, can be realized by the reward function $R(r)=2c\sigma^2\ln \zeta_\mu(q_\mu(r))+V_0$. We shall assume that we are in the nontrivial case $K>V_0$, otherwise the only attainable equilibrium is $\mu= \mathcal{N}(x_0, \sigma^2 T)$ which is induced by the uniform reward. We first relax the boundedness requirement of $\ln (f_\mu/f_0)$; it turns out that the the relaxed optimizer fails to be in $\mathcal{E}(\mathcal{R}_b^r)$. We then construct an approximate optimizer by truncation.

The relaxed optimization problem over $\mu$ reads
\begin{equation}\label{prob-effort}
\begin{aligned}
\text{maximize} \quad & \int_{-\infty}^\infty y f_\mu(y)dy-x_0\\
 \text{s.t.} \quad & \int_{-\infty}^\infty f_\mu(y)dy= 1, \quad \ln \left(\frac{f_\mu}{f_0}\right) \text{ is increasing},\\
&  \int_{-\infty}^{\infty} \ln \left(\frac{f_\mu(y)}{f_0(y)}\right) f_\mu(y) dy \le \frac{K-V_0}{2c\sigma^2}.
\end{aligned}
\end{equation}
Any candidate optimizer $f_{\mu^*}$ necessarily satisfies the Kuhn--Tucker conditions (see e.g.\ Luenberger \cite{Luenberger69})
\begin{align*}
&y=\lambda_1+\lambda_2\left(1+\ln \left(\frac{f_{\mu^*}(y)}{f_0(y)}\right)\right)\\
& \int_{-\infty}^\infty f_{\mu^*}(y)dy= 1, \quad  \int_{-\infty}^{\infty} \ln \left(\frac{f_{\mu^*}(y)}{f_0(y)}\right) f_{\mu^*}(y) dy \le \frac{K-V_0}{2c\sigma^2}, \quad \lambda_2\ge 0,\\
& \lambda_2\left( \int_{-\infty}^{\infty} \ln \left(\frac{f_{\mu^*}(y)}{f_0(y)}\right) f_{\mu^*}(y) dy-\frac{K-V_0}{2c\sigma^2}\right)=0.
\end{align*}
The above implies
\[f_{\mu^*}(y)
=\frac{1}{\sigma \sqrt{T}}\varphi\left(\frac{y-(x_0+\sigma^2 T/\lambda_2)}{\sigma\sqrt{T}}\right)\]
and
\[ \ln \left(\frac{f_{\mu^*}(y)}{f_0(y)}\right)=\frac{y-x_0}{\lambda_2}-\frac{\sigma^2 T}{2\lambda_2^2},\]
where $\lambda_2>0$ is determined by the complementary slackness
\[H\left(\mu^* | \mathcal{N}(x_0, \sigma^2 T)\right)=\frac{\sigma^2 T}{2\lambda_2^2}=\frac{K-V_0}{2c\sigma^2},\]
giving $\lambda_2=\sigma^2 \sqrt{cT/(K-V_0)}$.
We then have
\[f_{\mu^*}(y)=\frac{1}{\sigma \sqrt{T}}\varphi\left(\frac{y-(x_0+\sqrt{(K-V_0)T/c})}{\sigma\sqrt{T}}\right).\]
In other words, $\mu^*= \mathcal{N}(x_0+\sqrt{(K-V_0)T/c}, \sigma^2 T)$. It is also clear that $\ln f_{\mu^*}(y)/f_0(y)$ is increasing.
Since the objective and the equality constraints are linear in $f_\mu$, and the inequality constraint is convex in $f_\mu$, it can also be shown that these conditions, together with the monotonicity of $\ln f_{\mu^*}/f_0$, are sufficient for optimality. The relaxed optimal value equals $\sqrt{(K-V_0)T/c}\ge \sup_{R\in\mathcal{R}_b^r}A(R)$.

Since $\ln f_{\mu^*}/f_0$ is unbounded, such a $\mu^*\notin \mathcal{R}^r_b$. Consider the truncated $\mu_M$ defined in the theorem statement. We have $R_M\in\mathcal{R}_b^r$ and $\mu_M=\mathcal{E}(R_M)$ for all $M$. Moreover, let 
\[h_M(y):=f_0(y)\exp\left(\frac{y\wedge M\vee(-M)}{\lambda_2}-\frac{x_0}{\lambda_2}-\frac{\sigma^2 T}{2\lambda_2^2}\right).\]
We can show that
\[\int_{-\infty}^{\infty} h_M(y)dy=\int_{-\infty}^{\infty} f_{\mu^*}(y)dy+O(e^{-M/\lambda_2})=1+O(e^{-M/\lambda_2})\]
and
\[\int_{-\infty}^{\infty} yh_M(y)dy=\int_{-\infty}^{\infty}y f_{\mu^*}(y)dy+O(e^{-M/\lambda_2})=x_0+\sqrt{(K-V_0)T/c}+O(e^{-M/\lambda_2})\]
as $M\rightarrow \infty$.
It follows that
\begin{align*}
A(R_M)&= \frac{\int_{-\infty}^\infty y h_M(y) dy}{\int_{-\infty}^{\infty} h_M(z)dz}-x_0\\
&=\frac{x_0+\sqrt{(K-V_0)T/c}+O(e^{-M/\lambda_2})}{1+O(e^{-M/\lambda_2})}-x_0=\sqrt{(K-V_0)T/c}+O(e^{-M/\lambda_2}).
\end{align*} 
\end{proof}

\begin{remark}\label{rmk:opt-effort}
It can be verified using Theorem~\ref{thm:NE0} that $\mu^*=\mathcal{N}(x_0+\sqrt{(K-V_0)T/c}, \sigma^2 T)$ is an equilibrium induced by the unbounded reward
\begin{align*}
R^*(r)&=2c\sigma^2\ln \zeta_{\mu^*}(q_{\mu^*}(r))+V_0=2\sqrt{\frac{(K-V_0)c}{T}}(q_{\mu^*}(r)-x_0)-K+2V_0\\
&=K+2\sigma \sqrt{c(K-V_0)} N^{-1}(r).
\end{align*}
The optimal effort process associated with $\mu^*$ is constant: $a^*_s\equiv \sqrt{\frac{K-V_0}{cT}}$, by straightforward calculation using \eqref{eq:astar}. This can also be seen by directly substituting 
\[R^*_{\mu^*}(x)=2\sqrt{\frac{(K-V_0)c}{T}}(x-x_0)-K+2V_0\] 
into the control problem, yielding a linear-quadratic optimization:
\begin{align*}
&\sup_{a} E\left[R^*_{\mu^*}\left(x_0+\int_0^T a_t dt\right) - \int_0^T ca_t^2 dt\right]\\
&=\sup_{a} E\left[\int_0^T \left(2\sqrt{\frac{(K-V_0)c}{T}}a_t - ca_t^2\right) dt\right]-K+2V_0.
\end{align*}
However, it is not clear whether $\mu^*$ is the unique equilibrium under $R^*$.
\end{remark}

\section{Price of anarchy}\label{sec:PoA}

For a fixed reward function $R$, the \emph{price of anarchy} (PoA) is defined as the ratio between the optimal centralized welfare $V_c$ and the worst equilibrium welfare/game value.
By centralized, we mean that the principal can prescribe and enforce the effort, or equivalently, the law of the controlled process, for the agents. We only consider a symmetric effort prescription, i.e.\ same terminal law for all players. To avoid triviality, we consider $R$ that is not purely rank-based, otherwise the optimal centralized welfare is always equal to $\int_0^1 R(r)dr$ which is attained by prescribing zero effort for all.


The optimal centralized welfare $V_c$ is defined as
\[V_{c}:=\sup_{a} E \left[R_{\mathcal{L}(X_T)}(X_T)-\int_0^T ca_t^2 dt\right].\]
This is a control problem of McKean-Vlasov type.
Similar to the derivation of \eqref{eq:agent-problem2}, we can reformulate the centralized problem as
\begin{align*}
V_{c}
&=\sup_{\mu\in\mathcal{P}(\bbR): \mu\sim P_T}\int_\bbR \left\{R\left(x, F_\mu(x), \int_\bbR y d\mu(y)\right)-2c\sigma^2 \ln\left(\frac{d\mu}{d\mathcal{N}\left(x_0, \sigma^2 T\right)} \right)(x)\right\}d\mu(x)\\
&=\sup_{m\in \bbR}\sup_{\substack{\mu\in\mathcal{P}(\bbR): \mu\sim P_T  \\ \int_\bbR yd\mu(y)=m}}\int_\bbR \left\{R\left(x, F_\mu(x), m\right)-2c\sigma^2 \ln\left(\frac{d\mu}{d\mathcal{N}\left(x_0, \sigma^2 T\right)} \right)(x)\right\}d\mu(x).
\end{align*}

When $R(x,r,m)$ is independent of individual performance $x$, the inner optimization over $\mu$ is explicitly solvable. Specifically, letting $\Pi(m):=\int_0^1 R(r,m)dr$, we have
\begin{align*}
V_c
& =\sup_{m} \left\{\Pi(m)-2c\sigma^2 \inf_{\mu: \mu\sim P_T, \int_\bbR yd\mu(y)=m} \int_\bbR \ln\left(\frac{d\mu}{d\mathcal{N}\left(x_0, \sigma^2 T\right)} \right)(x)d\mu(x)\right\}.
\end{align*}
Here and in the sequel, we omit the underlying assumption that $m\in\bbR$ and $\mu\in\mathcal{P}(\bbR)$.
Using the Lagrange method, we find that the mean-constrained entropy minimization has optimal value $\frac{(m-x_0)^2}{2\sigma^2 T}$, attained by the normal distribution
$\mu^*=\mathcal{N}(m, \sigma^2 T)$. It follows that
\begin{equation}\label{eq:Vc1}
V_c= \sup_{m} \left\{\Pi(m)-\frac{c}{T}(m-x_0)^2\right\}.
\end{equation}
We see that $V_c<\infty$ if $\Pi(m)$ has sub-quadratic growth. As one would expect for a symmetric game, the centralized solution does not depend on the rank-order allocation of rewards.

When $R(x,r,m)$ is independent of rank $r$, we have
\begin{align*}
V_c&=\sup_{m}\sup_{\mu: \int_\bbR yd\mu(y)=m}\int_\bbR \left\{R\left(x, m\right)-2c\sigma^2 \ln\left(\frac{d\mu}{d\mathcal{N}\left(x_0, \sigma^2 T\right)} \right)(x)\right\}d\mu(x).
\end{align*}
Again by the Lagrange method, we find that the inner maximization over $\mu$ has solution
\[f_{\mu(m)}(x)=\frac{f_0(x)\exp\left(\frac{R(x,m)-\lambda_m x}{2c\sigma^2}\right)}{\int_\bbR f_0(y)\exp\left(\frac{R(y,m)-\lambda_m y}{2c\sigma^2}\right)dy},\]
where $\lambda_m$ is determined by
\[m=\frac{\int_\bbR xf_0(x)\exp\left(\frac{R(x,m)-\lambda_m x}{2c\sigma^2}\right)dx}{\int_\bbR f_0(y)\exp\left(\frac{R(y,m)-\lambda_m y}{2c\sigma^2}\right)dy}.\]
Plugging in $f_{\mu(m)}$ into the formula for $V_c$, we get
\begin{equation}\label{eq:Vc2}
V_c=\sup_{m}\left\{\lambda_m m +2c\sigma^2 \ln\left(\int_\bbR f_0(y)\exp\left(\frac{R(y,m)-\lambda_m y}{2c\sigma^2}\right)dy\right)\right\}.
\end{equation}


\begin{example}
Suppose $R(r,m)=m+2\sigma \sqrt{c(1-\alpha) m} N^{-1}(r)$ with $\alpha\in(0,1]$. Note that $R(r,m)$ takes the same form as the effort-maximizing reward in Remark~\ref{rmk:opt-effort} with $K$ replaced by $m$ and $V_0$ by $\alpha m$. In this case, we have $\Pi(m)=m$ and by \eqref{eq:Vc1},
\[V_c=\sup_{m} \left\{m-\frac{c}{T}(m-x_0)^2\right\}=x_0+\frac{T}{4c}.\]
To compute the equilibrium welfare, observe that given $\mu\in\mathcal{P}(\bbR)$ with $\int_\bbR x d\mu(x)=m_\mu$. Let $\tilde R(r):=R(r,m_\mu)$, then $\mu$ is an equilibrium for $R$ if and only if $\mu$ is optimal for $V(R,\mu)=V(\tilde R,\mu)$, which means $\mu$ is also the unique equilibrium for the purely rank-based reward $\tilde R$. 
This allows us to directly use Remark~\ref{rmk:opt-effort} to write down one equilibrium $\mu$ (not necessarily unique), whose mean satisfies
\[m_\mu=x_0+\sqrt{\frac{(1-\alpha) m_\mu T}{c}}.\]
The unique solution is given by
\[m_\mu=x_0+\frac{(1-\alpha)T}{2c}+ \frac{1}{2c}\sqrt{(1-\alpha)^2T^2+4cx_0(1-\alpha)T},\]
with associated game value $\alpha m_\mu$. It follows that in this case, 
\[\text{PoA}\ge \frac{x_0+\frac{T}{4c}}{\alpha m_\mu}.\]
When $\alpha\rightarrow 0$, PoA$\rightarrow \infty$. When $\alpha=1$, PoA$\ge 1+\frac{T}{4cx_0}$. When $\alpha m_\mu=x_0+\frac{T}{4c}$ or $\alpha=\frac{T+4cx_0}{2T+4cx_0}$, PoA$\ge1$.
If we only consider equilibria that satisfy $\beta(\mu)<\infty$, then all inequalities become equalities.
\end{example}

\begin{example}
Suppose $R(x,m)=\alpha x+(1-\alpha)g(m)$ where $g$ is bounded increasing, and $\alpha\in[0,1]$. In this case, it can be shown that $\mu(m)=\mathcal{N}(m,\sigma^2 T)$, $m=x_0+\frac{T}{2c}(\alpha-\lambda_m)$, and
\[\int_\bbR f_0(y)\exp\left(\frac{R(y,m)-\lambda_m y}{2c\sigma^2}\right)dy=\exp\left(\frac{(1-\alpha)\Pi(m)+\frac{c}{T}(m^2-x_0^2)}{2c\sigma^2}\right).\]
By \eqref{eq:Vc2}, we get 
\[V_c=\sup_{m}\left\{\alpha m+(1-\alpha)g(m)-\frac{c}{T}(m-x_0)^2\right\}.\]
Since $R$ is independent of rank and linear in $x$, we have that $\beta(\mu)<\infty$ for all $\mu\in\mathcal{P}(\bbR)$. By Theorem~\ref{thm:NE0}, all equilibria are characterized by \eqref{eq:fixed-point}. 
We find that there is a unique equilibrium $\mu$ which is normal with mean $m_{\mu}=x_0+\frac{T\alpha}{2c}$ and variance $\sigma^2T)$. The associated game value is
\[2c\sigma^2 \ln\left(\int_\bbR f_0(y)\exp\left(\frac{R(y,m_{\mu})}{2c\sigma^2}\right)dy\right)=(1-\alpha)g(m_{\mu})+\frac{c}{T}(m_{\mu}^2-x_0^2).\]
It follows that, after rearranging the denominator,
\begin{align*}
\text{PoA}
&=\frac{\sup_{m}\left\{\alpha m+(1-\alpha)g(m)-\frac{c}{T}(m-x_0)^2\right\}}{\alpha m_{\mu}+(1-\alpha)g(m_{\mu})-\frac{c}{T}(m_{\mu}-x_0)^2}.
\end{align*}
When $\alpha=0$, PoA=$\sup_{m}\left\{g(m)-\frac{c}{T}(m-x_0)^2\right\}/g(x_0)$. When $\alpha=1$, the game is non-interactive, and PoA$=1$.
\end{example}

\appendix
\section{Schr\"odinger bridges from space to time}\label{sec:Appen}

Let $\Omega=C([0,T],\bbR)$ be the canonical space and $\bbW_x$ be the Wiener measure starting at $x$ at time zero. Also let $(\mathcal{F}_t)_{t\in[0,T]}$ be the filtration generated by the canonical process. Define $\tau(\omega):=\inf\{t\in[0,T]: w_t=0\}$ with the convention that $\inf \emptyset=\infty$. Given a reference measure $P\in\mathcal{P}(\Omega)$, a source distribution $\nu\in\mathcal{P}(\bbR)$ and a target distribution $\mu\in\mathcal{P}(\bbT)$ where $\bbT:=[0,T]\cup\{\infty\}$, consider the following variant of the Schr\"odinger bridge problem:
\begin{equation*} 
\inf_{Q\in\mathcal{P}(\Omega)} H(Q|P)\quad \text{subject to} \quad Q_0=\nu, \quad Q\circ \tau^{-1}=\mu.
\end{equation*}
For any $Q\in\mathcal{P}(\Omega)$, define $Q^{x,t}:=Q(\cdot| \omega_0=x, \tau(\omega)=t)$. We have the disintegration:
\[Q(\cdot)=\int_{\bbR\times \bbT} Q^{x,t}(\cdot)\pi(dx,dt), \quad \pi=Q\circ (\omega_0, \tau)^{-1}.\]
Similar to the standard Schr\"odinger bridge problem, one can show that the optimal transport plan is given by
\[Q^*=\int_{\bbR\times \bbT} P^{x,t}(\cdot)\pi^*(dx,dt),\]
where $\pi^*$ is the solution to
\[ \inf_{\pi \in\mathcal{P}(\bbR\times \bbT)} H(\pi | P\circ (\omega_0, \tau)^{-1})\quad \text{subject to} \quad \pi_0=\nu, \quad \pi_1=\mu,\]
assuming the infimum is attained. 

Now, consider the hitting time ranking game of Bayraktar et al.\ \cite{BCZ18}, where each agent solves
\begin{equation*}\label{eq:hitting-time-game}
\hat V=\sup_aE\left[R_{\tilde \mu}(\tau)-\int_0^{\tau\wedge T} ca_t^2 dt\right], \quad \tau=\inf\left\{t\ge 0: x_0-\int_0^t a_s ds+\sigma B_s=0\right\}.
\end{equation*}
Here it is assumed that $R_{\tilde\mu}(t)=R_\infty\in\bbR$ for all $t>T$.
Take $X=x_0+\sigma \omega_t$ and $P=\mathbb{W}_0\circ X^{-1}$, and identify $a_t$ with the set of laws
\[\mathcal{Q}:=\left\{Q\in\mathcal{P}(\Omega): Q\sim P, \, \frac{dQ}{dP}=\frac{dQ}{dP}\bigg|_{\mathcal{F}_{\tau\wedge T}}\right\}.\] 
The condition on the Radon-Nikodym derivative means $a_t\equiv 0$ for all $t>\tau\wedge T$.
Let $\mu_0:=P\circ \tau^{-1}$ be the law of the first passage time of level $x_0/\sigma$ of a Brownian motion.
We can rewrite the agent's control problem in weak formulation as
\begin{align*}
\hat V&=\sup_{Q\in\mathcal{Q}} E^Q R_{\tilde \mu}(\tau)-2c\sigma^2 H(Q|P)\\
&=\sup_{\mu\sim \mu_0}\sup_{Q\in\mathcal{Q}:\, Q\circ \tau^{-1}=\mu} \int_\bbT R_{\tilde \mu}(x)d\mu(x)-2c\sigma^2 H(Q|P)\\
&\le\sup_{\mu\sim \mu_0} \int_\bbT R_{\tilde \mu}(t)d\mu(t)-2c\sigma^2 \inf_{Q\in\mathcal{P}(\Omega): \,Q_0=\delta_{x_0}, \,Q\circ \tau^{-1}=\mu} H(Q|P)\\
&=\sup_{\mu\sim \mu_0} \int_\bbT R_{\tilde \mu}(t)d\mu(t)-2c\sigma^2 \inf_{\pi \in\mathcal{P}(\bbR\times\bbT): \,\pi_0=\delta_{x_0}, \,\pi_1=\mu} H(\pi |P\circ (\omega_0, \tau)^{-1})\\
&=\sup_{\mu\sim \mu_0} \int_\bbT R_{\tilde \mu}(t)d\mu(t)-2c\sigma^2 H(\mu | \mu_0)
\end{align*}
Note that for each $\mu\sim \mu_0$, the associated optimal $Q=Q_\mu=\int_\bbT P^{x_0,t}(\cdot)\mu(dt)$ is always equivalent to $P$ and satisfies
$dQ/dP=\zeta(\tau)\in\mathcal{F}_{\tau\wedge T}$, where $\zeta(t):=d\mu(t)/d\mu_0(t)$. It follows that $Q\in\mathcal{Q}$ and the inequality is in fact an equality. The resulting static problem can be further split into a constrained calculus of variation problem, followed by a static optimization:
\begin{align*}
\hat V=\sup_{\beta\in (0,1)}\sup_{f_\mu|_{[0,T]}>0, \,\int_0^T f_\mu(t)dt=\beta}& \int_0^T \left\{R_{\tilde \mu}(t)-2c\sigma^2 \ln\left(\frac{f_\mu(t)}{f_{\mu_0}(t)}\right)\right\} f_\mu(t)dt\\
& + R_\infty(1-\beta)-2c\sigma^2 (1-\beta)\ln\left(\frac{1-\beta}{1-F_{\mu_0}(T)}\right).
\end{align*}
By elementary calculation, one deduces the same formula as Bayraktar et al.\ \cite[Eq.\ (2.5)]{BCZ18}.

\bibliographystyle{plain}
\bibliography{tournament}{}

\end{document}